\documentclass[12pt]{amsart}
\usepackage[utf8]{inputenc}
\usepackage{enumitem}
\usepackage{graphicx}
\usepackage{amsmath,amsfonts,latexsym}

\renewcommand*{\d}{\mathop{}\!d}
\DeclareMathOperator*{\id}{Id}

\author{Petar Melentijević}
\thanks{The author is partially supported by MPNTR grant 174017, Serbia}
\title[On seminorms of the weighted Berezin transform]{On the Bloch-type seminorms of the weighted Berezin transform}
\date{}
\address{Faculty of Mathematics \endgraf University of Belgrade \endgraf Studentski trg 16 \endgraf 11000 Beograd, Serbia}
\email{petarmel@matf.bg.ac.rs}
\subjclass[2010]{Primary 45P05; Secondary 31A10}
\keywords{Berezin transform, Operator norm, Norm estimate, Bloch space}

\newtheorem*{theorem*}{Theorem}
\newtheorem*{corollary*}{Corollary}
\newtheorem*{lemma*}{Lemma}
\newtheorem{theorem}{Theorem}
\newtheorem{lemma}{Lemma}

\theoremstyle{definition}

\theoremstyle{remark}

\begin{document}
\begin{abstract}
  We consider a weighted Berezin transform in the unit ball $\mathbb{B}^n \subset \mathbb{C}^n:$
  $$ B_{\alpha} : L^{\infty} (\mathbb{B}^n) \to \ \mathcal{B},\quad  \alpha>-1,$$
  defined, for $f \in L^{\infty} \left( \mathbb{B}^n \right)$ and $z \in \mathbb{B}^n$, by
  $$ (B_{\alpha} f) (z) = c_{\alpha} \int_{\mathbb{B}^n} \frac{\left( 1-|z|^2
    \right)^ {n+1}}{|1 - \langle z, w \rangle|^{2n+2}} f(w) \left( 1-|w|^2
  \right)^\alpha \d v(w),$$
  where $c_{\alpha} = \frac{\Gamma(\alpha+n+1)}{\Gamma(\alpha+1)\pi^n}$ , $v$
  is the Lebesque measure and $\mathcal{B}$ is $C^1$ Bloch-type space. We prove that
  $B_{\alpha}$ is bounded iff $\alpha \geq0$ and give the exact complex Bloch seminorm of $ B_\alpha$ for $0\leq\alpha \leq 2n+3.$ We also consider real Bloch seminorm and give sharp estimate for $0 \leq \alpha \leq n+\frac{1}{2}.$
\end{abstract}

\maketitle

\section{Introduction}

We use the notation from Rudin's monograph \cite{Rud}. Throughout the
paper $n$ is a positive integer. We denote the inner product in $\mathbb{C}^{n}$
by:
$$\langle z, w \rangle = z_1 \overline{w}_1 + z_2\overline{w}_2 + \dots
+ z_n \overline{w}_n,$$
where $z,w \in \mathbb{C}^n$. This inner product induces the Euclidean norm
$$ |z| = \sqrt{\langle z,z \rangle}.$$
Let $\mathbb{B}^n$ be the unit ball $\{z \in \mathbb{C}^n | |z| < 1 \}$ and  $\mathbb{B}_n$ its real counterpart. Let $e_1,e_2,\dots e_n$ be the standard base for $\mathbb{C}^n$.

We let $v$ be the volume measure in $\mathbb{C}^n$. We will also consider a
class of weighted volume measures on $\mathbb{B}^n$. For $\alpha > -1$ we define
a finite measure $v_{\alpha}$ on $\mathbb{B}^n$ by
$$\d v_{\alpha} (z) = c_{\alpha} \left(1- |z|^2 \right)^{\alpha}\d v(z),$$
where $c_{\alpha}$ is a normalizing constant such that $v_{\alpha}(\mathbb{B}^n)
=1$. Using polar coordinates, one can easily calculate that
\begin{equation}
    c_{\alpha} = \frac{\Gamma(\alpha+n+1)}{\Gamma(\alpha+1)\pi^n}.
\end{equation}

We will often use the following automorphisms of $\mathbb{B}^n$:
\begin{equation}
  \varphi_z (\xi) =  \frac{1}{1 - \langle \xi , z \rangle}
  \left( z - \frac{\langle \xi, z \rangle}{|z|^2} z -
    \left( 1 - |z|^2 \right)^{\frac{1}{2}} \left(\xi  -
      \frac{\langle \xi , z \rangle}{|z|^2} z \right) \right).
\end{equation}
Observe that $\varphi_z (0) = z,$ and since $ \varphi_z$ is involutive i.e. $\varphi_z
\circ \varphi_z = \id_{\mathbb{B}^n},$ we also have $\varphi_z(z)=0$. We will also use
the following identities
\begin{align}
  1-|\varphi_z(\xi)|^2 &= \frac{\left(1 - |z|^2 \right)
                      \left(1- |\xi|^2 \right)}{|1 - \langle z, \xi \rangle|^2}, \\
  1 - |z|^2 &= (1 - \langle z, \xi \rangle)
              \left(1 - \left\langle z, \varphi_z(\xi) \right\rangle\right),
\end{align}
for $z$, $\xi \in \mathbb{B}^n$, see \cite{Rud} for more details.

The real Jacobian of $\varphi_z$ is given by
$$ (J_{\mathbb{R}}  \varphi_z)(\xi) =
\left(\frac{(1 - |z|^2)}{|1 - \langle z, \xi \rangle|^2}\right)^{n+1}, \quad
z,\xi \in \mathbb{B}^n.$$
For a function $f \in C^1(\mathbb{B}^n)$, we define  complex gradients
\begin{align*}
  \nabla_{z} f(z) &= \left(\frac{\partial f(z)}{\partial z_1},
                    \frac{\partial f(z)}{\partial z_2}, \dots,
                    \frac{\partial f(z)}{\partial z_n} \right), \\
  \nabla_{\overline{z}} f(z) &= \left(\frac{\partial f(z)}{\partial \overline{z_1}},
                               \frac{\partial f(z)}{\partial \overline{z_2}},
                               \dots,
                               \frac{\partial f(z)}{\partial \overline{z_n}} \right), \\
  \intertext{and real gradient}
  \nabla f(z) &= \left(\frac{\partial f}{\partial x_1},
                \frac{\partial f}{ \partial y_1},
                \dots, \frac{\partial f}{\partial x_n},
                \frac{\partial f}{ \partial y_n} \right),
\end{align*}
where $z_k = x_k+ iy_k$, $k = \overline{1,n}$.
 
It is not hard to see that
$$ |\nabla f|^2 = 2\left(|\nabla_{z} f|^2 + |\nabla_{\overline{z}} f|^2\right).$$

Let us define a $C^1$ Bloch-type space as:
$$\displaystyle \mathcal{B}= \{ f \in C^1(\mathbb{B}^n)| \|f\|_*=\sup_{|z|<1}
(1-|z|^2)|\nabla f(z)| < +\infty \}.$$
 This is a real variable analogue of the
classical analytic Bloch space. $\mathcal{B}$ is a Banach space with the norm
$$\| f \|_ {\mathcal{B}} = |f(0)|+ \|f\|_* ,\quad f\in \mathcal{B}. $$
As the standard reference for the Bloch space we refer to \cite{Tim1} and \cite{Tim2}.

Berezin transform is an integral operator acting on functions defined on
the unit ball $\mathbb{B}^n \subset \mathbb{C}^n$. For a function $ f:
\mathbb{B}^n \to \mathbb{C}$ its Berezin transform is a new function
$$Bf : \mathbb{B}^n \to \mathbb{C}$$
defined at a point $z\in \mathbb{B}^n$ by
$$ (Bf) (z) = \int_{\mathbb{B}^n} \frac{(1 - |z|^2)^{n+1}}{|1 - \langle z, w
  \rangle|^{2n+2}} f(w) \d v(w). $$

Berezin \cite{Ber} introduced the notion of covariant and contravariant symbols
of an operator. Successfull applications of Berezin transform are so far mainly
in the study of Hankel and Toeplitz operators. It should be pointed out that the
Berezin transform is an analogue of the Poisson transform in Bergman space
theory, see \cite{Zhu}. Recent results concerning the norm of the Berezin
transform on $L^p (\mathbb{B}^n)$ can be found in Dostanić's and Marković's
papers \cite{Dos} and \cite{Mark1}.

Along with the $L^p-$norm estimates, in recent years there has been increased interest in exploring the magnitude of certain operators in terms of other operator norms. Here, we consider Berezin transform as $L^{\infty} \to \mathcal{B}$ operator. The method developed in \cite{Petar} is used in the present setting as well.In part this work is motivated by papers \cite{Kal} and \cite{Per}, where analogous problems were investigated for the Bergman projection.

Let $f \in L^{\infty} (\mathbb{B}^n)$ and let consider a slightly general situation -- weighted Berezin transform on
$\mathbb{B}^n$ , i.e.
\begin{equation}
  (B_{\alpha} f) (z) =  \int_{\mathbb{B}^n} \frac{(1- |z|^2)^ {n+1}}
  {|1 - \langle z, w \rangle|^{2n+2}} f(w) \d v_{\alpha} (w),\quad  f \in L^{\infty}(\mathbb B^n),z \in \mathbb B^n.
\end{equation}

We investigated two seminorms of the operator $B_\alpha:$
$$ \|B_\alpha\|_{*,\mathbb{C}}= \sup\limits_{\|f\|_{\infty} \leq 1} {(1-|z|^2)|\nabla_z(B_\alpha f)(z)| }$$
and
$$ \|B_\alpha\|_*= \sup\limits_{\|f\|_{\infty} \leq 1} {\|B_\alpha f \|_*}.
$$

Now we state the main results of this paper.
\begin{theorem}
Let $\alpha > -1$ ,  $B_{\alpha}$ is an operator defined by $(5)$ and $ f \in
L^{\infty}(\mathbb{B}^n).$ Then the following holds:
\begin{enumerate}[label=\roman*)]
\item If $0 \leq \alpha \leq 2n+3$ and $ \|f\|_\infty= 1$ , then
  \begin{align*}
    \|B_\alpha\|_{*,\mathbb{C}}= c_{\alpha} (n+1)
    \int_{\mathbb{B}^n} |\zeta_1| \left(1 - |\zeta|^2 \right)^{\alpha} \d
    v(\zeta) = \\
    = \frac{n+1}{2}B(n+ \alpha+1,\frac{1}{2}),
  \end{align*}
    where the equality is attained if and only if  $f(w)= C \dfrac{|\langle w,a \rangle|}{\overline{\langle w,a \rangle}} $ for some $a \in \mathbb{C}^n, |a|=1$ and $C$ is an unimodular constant.
   
    \item If $-1< \alpha < 0$, then $\|B_\alpha\|_{*,\mathbb{C}} = + \infty.$
    \item If $\alpha>2n+3$ and $\|f\|_{\infty}=1$, then
      $$\|B_\alpha\|_{*,\mathbb{C}} < \, (n+1)\Gamma(n+\alpha +1)
      \binom{\alpha+ k_{\alpha} - 1}{k_{\alpha}}
      \frac{\Gamma{\left(k_{\alpha} + \frac{3}{2} \right)}}
      {\Gamma{\left(k_{\alpha} + \alpha + n + \frac{3}{2}\right)}},$$
      where
      $$k_{\alpha} = \left\lceil \frac{\alpha - (2n+3)}{2n+2} \right\rceil $$
      and $\lceil x \rceil$ is the smallest non-negative integer not smaller than $x$.

The same estimates holds for conjugate derivative seminorm.
      
  \end{enumerate}
\end{theorem}
  
  \begin{theorem}
  	For $f$ real valued, under the same hypothesis as in the Theorem 1, we have:
  	\begin{enumerate}[label=\roman*)]
  		\item If $0 \leq \alpha \leq n+\frac{1}{2}$ and $ \|f\|_\infty= 1$ , then
  		\begin{align*}
  			 \|B_\alpha\|_*=2c_{\alpha} (n+1)
  			\int_{\mathbb{B}^n} | Re \zeta_1| \left(1 - |\zeta|^2 \right)^{\alpha} \d
  			v(\zeta) = \\
  			= \frac{2}{\pi}(n+1)B(\alpha+n+1,\frac{1}{2}),
  		\end{align*}
  		where the equality is attained if and only if  $f(w)= \dfrac{| Re \langle w,a \rangle|}{ Re\langle w,a \rangle} $ for some $a \in \mathbb{C}^n, |a|=1.$
  		
  		\item If $-1< \alpha < 0$, then $ \|B_\alpha\|_* = + \infty.$
  		\item If $\alpha>n+\frac{1}{2}$ and $\|f\|_{\infty}=1$, then
  		$$ \|B_\alpha\|_* < \frac{2}{\sqrt{\pi}}(n+1)\binom{2k'_{\alpha}+2\alpha-1}{2k'_{\alpha}} {\binom{k'_{\alpha}+\alpha-1}{k'_{\alpha}}}^{-1} \frac{(k'_{\alpha})!\Gamma(\alpha+n+1)}{\Gamma(k'_{\alpha}+\alpha+n+\frac{3}{2})},$$
  		where
  		$$k'_{\alpha} = \left\lceil \frac{\alpha}{2n+1}-\frac{1}{2}  \right\rceil $$
  		and $\lceil x \rceil$ is the smallest non-negative integer not smaller than $x$.
  		
  	\end{enumerate}
  \end{theorem}
  
\section{Proof of the Theorem 1}

In order to find partial derivatives of $B_{\alpha} f(z)$ we need formulae for
$\frac{\partial K}{\partial z_i}$ where
$K(z, w) = \dfrac{(1 - |z|^2)^{n+1}}{|1 - \langle z, w \rangle|^{2n+2}}$. Since
$$K(z,w) = \frac{(1 - |z|^2)^{n+1}}{|1 - \langle z, w \rangle|^{2n+2}}
= \frac{ \left(1 - \sum{z_i \overline{z}_i} \right)^{n+1}}{ \left(1 - \sum {z_i
      \overline{w}_i} \right)^{n+1} \left( 1 - \sum{\overline{z}_i w_i} \right)^{n+1}},$$
we have
\begin{align*}
  \frac{\partial K}{\partial z_i}(z,w) &= \frac{1}
                                         {\left(1 - \sum \overline{z}_i w_i \right)^{n+1}}
                                         \frac{\partial}{\partial z_i}{
                                         \left(\frac{1 - \sum{z_i \overline{z}_i}}
                                         {1 - \sum {z_i \overline{w}_i}} \right)^{n+1}} \\
                                       &= \frac{n+1}
                                         {\left(1 - \sum{\overline{z}_i w_i} \right)^{n+1}}
                                         \left( \frac{1 - \sum{z_i \overline{z}_i}}
                                         {1 - \sum {z_i \overline{w}_i}} \right)^{n}
                                         \frac{\partial}{\partial z_i}{
                                         \left(\frac{1 - \sum{z_i \overline{z}_i}}
                                         {1 - \sum{z_i\overline{w}_i}} \right)} \\
                                       &= \frac{(n+1) (1- |z|^2)^n
                                         \left( -\overline{z}_i
                                         \left(1 - \sum{z_i \overline{w}_i} \right)
                                         + \overline{w}_i
                                         \left(1 -\sum{z_i \overline{z}_i} \right)\right)}
                                         {\left(1 - \sum{w_i \overline{z}_i} \right)^{n+1} \left(1 - \sum{z_i \overline{w}_i} \right)^n
                                         \left(1 - \sum{z_i \overline{w}_i} \right)^2} \\
                                       &= \frac{(n+1)(1-|z|^2)^n}
                                         {(1 - \langle z,w \rangle)^{n+2}
                                         (1 - \overline{\langle z , w \rangle})^{n+1}}
                                         \left( (1 - |z|^2)\overline{w}_i
                                         - (1 - \langle z, w \rangle)\overline{z}_i \right).
\end{align*}
Therefore:
\begin{gather*}
\begin{split}
  |\nabla_z (B_{\alpha}f)(z)| &= \sup_{ |\xi|=1}
                                |\langle \nabla_z (B_{\alpha}f)(z), \xi \rangle| \\
                              &= \sup_{ |\xi|=1}
                                \left| \int_{\mathbb{B}^n} \langle \nabla_{z} K (z,w)f(w), \xi  \rangle \d v_{\alpha}(w) \right| \\
                              & \leq \sup_{|\xi|=1} \int_{\mathbb{B}^n}
                              |\langle \nabla_z K(z,w), \xi \rangle| \, |f(w)|
                              \d v_{\alpha} (w) =
                            \end{split} \\
                            \resizebox{0.99\textwidth}{!}{%
                             $\displaystyle = \sup_{ |\xi|=1} \int_{\mathbb{B}^n}
                              \left| \sum_{i=1}^n \frac{(n+1)(1- |z|^2)^n \left(
                                  (1-|z|^2)\overline{w}_i\overline{\xi}_i - (1-
                                  \langle z, w \rangle) \right)\overline{z}_i\overline{\xi}_i }{(1 -
                                  \sum{\overline{z}_i w_i})^{n+1} (1 -
                                  \sum{z_i \overline{w}_i})^{n+2}} \right|
                              |f(w)| \d v_{\alpha}(w)$} \\
                            \resizebox{\textwidth}{!}{%
                              $\displaystyle = \sup_{ |\xi|=1} \int_{\mathbb{B}^n} \frac{(n+1)(1-|z|^2)^n}{|1 - \langle z, w \rangle|^{2n+3}} \left|
                                (1-|z|^2)\langle \xi, w \rangle - (1-\langle z,w
                                \rangle)\langle \xi,z \rangle \right| |f(w)| \d
                              v_{\alpha} (w).$}
\end{gather*}

    Let us denote $S(z)=(1- |z|^2)|\nabla_{z} (B_{\alpha}f) (z)|.$ For $\|f\|_{\infty} \leq 1$,we obtained:
\begin{align*}
    S(z) \leq (n+1)\sup_{ |\xi|=1}\int_{\mathbb{B}^n}
    &\frac{(1-|z|^2)^{n+1}}{|1- \langle z, w \rangle|^{2n+3}} \cdot {} \\ 
    &{} \cdot \big| (1- |z|^2)\langle \xi , w \rangle - (1 - \langle z , w
    \rangle)\langle \xi, z \rangle  \big| \d v_{\alpha} (w).
\end{align*}
    In the above integral, we introduce new variable: $\zeta = \varphi_{z} (w)$
    (or: $w = \varphi_z (\zeta)$) and this gives:
    \begin{align*}
      S(z) \leq (n+1) c_\alpha \sup_{ |\xi|=1} \int_{\mathbb{B}^n}
      &\frac{|(1 - |z|^2)\langle \xi, \varphi_z(\zeta) \rangle - (1 - \langle z,
        \varphi_{z} (\zeta)\rangle)\langle \xi, z \rangle|}{|1- \langle z ,
        \varphi_{z} (\zeta)\rangle|} \cdot {} \\
      &{} \cdot (1-|\varphi_z(\zeta)|^2)^{\alpha} \d v(\zeta).
      \end{align*}
    Using $(4)$ we obtain
    $$\frac{(1 - |z|^2) \langle \xi , \varphi_z(\zeta) \rangle}{1 - \langle z , \varphi_z(\zeta)\rangle} = (1 - \langle z, \zeta \rangle) \langle \xi, \varphi_z(\zeta)  \rangle, $$
    therefore
    \begin{multline*}
      \bigg| \frac{(1 - |z|^2)\langle \xi, \varphi_z (\zeta)  \rangle - (1 -
        \langle z, \varphi_z (\zeta)   \rangle) \langle \xi, z \rangle}{1 - \langle
        z, \varphi_z (\zeta) \rangle}       \bigg| = \\
    = \big| (1 - \langle z, \zeta \rangle) \langle \xi, \varphi_{z}(\zeta) \rangle - \langle \xi, z \rangle    \big| = 
     \langle \xi, (1 - \overline{ \langle z, \zeta \rangle}) \varphi_{z} (\zeta) -
     z\rangle.
     \end{multline*}
    Next, $(2)$ gives us:
    $$ (1 - \overline{\langle z , \zeta \rangle}) \varphi_{z} (\zeta) = z - \frac{\langle \zeta, z \rangle}{|z|^2} z - \sqrt{1 - |z|^2}(\zeta - \frac{\langle \zeta, z \rangle}{|z|^2}z), $$
    and therefore:
    \begin{multline*}
      \bigg| \frac{(1 - |z|^2)\langle \xi, \varphi_z (\zeta)  \rangle - (1 -
        \langle z, \varphi_z (\zeta)   \rangle) \langle \xi, z \rangle}{1 - \langle
        z, \varphi_z (\zeta) \rangle}       \bigg| =\\
    =\big| \langle \xi, \frac{\langle \zeta, z\rangle}{|z|^2}z + \sqrt{1 -
      |z|^2}(\zeta - \frac{\langle \zeta , z \rangle}{|z|^2}z)    \rangle \big|
    \end{multline*}
    for $z \neq 0$;
     for $z=0$ this expression is equal to
    $$ \big| \langle \xi, (1  - \overline{\langle z, \zeta \rangle} \varphi_z (\zeta) - z \rangle)    \big| = |\langle \xi, \zeta   \rangle|. $$
    Now having all this in mind, we have:
    \begin{align*}
     S(z) \leq (n+1)c_\alpha \sup_{ |\xi|=1} \int_{\mathbb{B}^n} &\left| \langle \xi , \frac{\langle \zeta, z\rangle}{|z|^2}z + \sqrt{1 - |z|^2} (\zeta - \frac{\langle \zeta, z \rangle}{|z|^2}z)\rangle     \right| \cdot {} \\
      &{} \cdot (1 - |\varphi_z(\zeta)|^2)^{\alpha} \d v(\zeta)
      \end{align*}
    
    Now, using $(3)$ we get:
    \begin{equation}
      \begin{aligned}
      S(z) \leq (n+1) c_\alpha \sup_{ |\xi|=1} \int_{\mathbb{B}^n} &\big|\langle \xi , \frac{\langle \zeta, z\rangle}{|z|^2}z + \sqrt{1 - |z|^2} (\zeta - \frac{\langle \zeta, z \rangle}{|z|^2}z)\rangle \big| \cdot {} \\ 
      &{} \cdot \frac{(1- |z|^2)^{\alpha}(1- |\zeta|^2)^{\alpha}}{|1 - \langle z. \zeta  \rangle|^{2\alpha} } \d v(\zeta)
      \end{aligned}
    \end{equation}
    for $z \neq 0$ or, for $z=0$:
    $$ S(z) \leq (n+1) c_\alpha \sup_{ |\xi|=1} \int_{\mathbb{B}^n} |\langle \xi, \zeta  \rangle|(1- |\zeta|^2)^{\alpha} \d v(\zeta) $$
    Without loss of generality, we can assume $z=(r,0,\dots,0)$, where $0 \leq r<1.$
    \par This gives us:
    $$S(z) \leq (n+1) c_{\alpha} T(r), $$
    where
\begin{align*}
T(r) &=  \sup_{ |\xi|=1} \int_{\mathbb{B}^n} \left| \vphantom{\sqrt{1-r^2}} \langle \xi, (\zeta_1, 0, \dots , 0) + {} \right. \\
    &\hspace{8em} \left. {}+\sqrt{1-r^2}(0, \zeta_2,\dots,\zeta_{n}) \rangle \right| \frac{(1-r^2)^{\alpha}(1 - |\zeta|^2)^{\alpha}}{|1-r\zeta_1|^{2\alpha} } 
    \d v(\zeta) \\
     &= \max_{ |\xi|=1} (1-r^2)^{\alpha} \int_{\mathbb{B}^n} \left|\langle \xi, \sqrt{1-r^2}\zeta + {} \right. \\
  &\hspace{12em} \left. {} + (1-\sqrt{1-r^2})\zeta_1^{'} \rangle \right| \frac{(1- |\zeta|^2)^{\alpha}}{{|1-r\zeta_1|^{2\alpha}}} \d v(\zeta),
\end{align*}
    where $\zeta_1^{'} = (\zeta_1, \dots, 0)$ and $\zeta_1$ is the first  complex coordinate of $\zeta = (\zeta_1, \zeta_2,\dots, \zeta_n) \in \mathbb{B}^n$.
 
    \par We have:
    \begin{equation}\begin{aligned}
      \big| \langle \xi, \sqrt{1-r^2}\zeta +(1-\sqrt{1-r^2})\zeta_{1}^{'} \rangle  \big| \\
    &\hspace{-4em}\leq \sqrt{1-r^2} | \langle \xi, \zeta \rangle \big| + (1-\sqrt{1-r^2})\big| \langle \xi, \zeta_{1}^{'} \rangle \big|. 
    \end{aligned}\end{equation}
  \par So, for all $r \in [0,1)$ we have, using triangle inequality $(7)$
  \begin{equation}\begin{split}
      &\begin{split}
        T(r)=(1-r^2)^{\alpha} \max_{ |\xi|=1} \int_{\mathbb{B}^n } \big|\langle \xi, \sqrt{1-r^2}\zeta + (1-\sqrt{1-r^2})\zeta_1^{'} \rangle \big| \cdot {} \\
        {} \cdot \frac{(1- |\zeta|^2)^{\alpha}} {\big| 1 - r\zeta_{1}   \big|^{2\alpha}}\d v(\zeta)\end{split} \\
    &\quad\begin{split}
    &\leq (1-\sqrt{1-r^2})(1-r^2)^{\alpha} \max_{ |\xi|=1} \int_{\mathbb{B}^n} \big| \langle \xi, \zeta_1^{'}  \rangle \big| \frac{(1-|\zeta|^2)^{\alpha}}{\big|1- r\zeta_1  \big|^{2\alpha}}\d v(\zeta) + {} \\
     &\quad {} + \sqrt{1-r^2}(1-r^2)^{\alpha} \max_{ |\xi|=1} \int_{\mathbb{B}^n} \big| \langle \xi, \zeta \rangle  \big| \frac{(1 -|\zeta|^2)^{\alpha}}{\big| 1 - r\zeta_1  \big|^{2\alpha}}\d v(\zeta).\end{split} 
     \end{split}\end{equation}
    The first integral in $(8)$ is easy to estimate, because $ \big|\langle \xi,
    \zeta_{1}^{'}\rangle \big| = \big| \xi_1 \zeta_1   \big| \leq \big| \zeta_1
    \big|$, and the equality is achieved for $ \xi = (1,0,\dots,0)$, i.e.
    \begin{multline*}
      (1-\sqrt{1-r^2})(1-r^2)^{\alpha}\max_{ |\xi|=1} \int_{\mathbb{B}^n}
      \big| \langle \xi, \zeta_{1}^{'} \rangle \big| \frac{(1-
        |\zeta|^2)^{\alpha}}{\big|1 - r\zeta_1 \big|^{2\alpha}} dv(\zeta)= \\
        = (1-\sqrt{1-r^2})(1-r^2)^{\alpha} \int_{\mathbb{B}^n} \frac{|\zeta_1|
         (1 - |\zeta|^2)^{\alpha}}{|1- r\zeta_1|^{2\alpha}} \d v(\zeta).
       \end{multline*}
       \par In order to estimate the second integral,  we need the following lemma:
       \begin{lemma}
         \begin{multline*}
           \max_{ |\xi|=1}  \int_{\mathbb{B}^n}|\langle \xi, \zeta \rangle | \frac{(1- |\zeta|^2)^{\alpha}}{{|1-r\zeta_1|^{2\alpha}}} dv(\zeta)= \\
      =\max_{|\xi_1|^2+|\xi_2|^2=1 }  \int_{\mathbb{B}^n}|\zeta_1 | \frac{(1-
        |\zeta|^2)^{\alpha}}{{|1-r\zeta_1\xi_1-r\zeta_2\xi_2|^{2\alpha}}}
     \d v(\zeta).
      \end{multline*}
    \end{lemma}
    \begin{proof}
       Note that there exits a unitary change of variable such that
         $$ Ue_1=e_1, \quad U\xi=\xi', $$ where $\xi'$ satisfies conditions
         $$\langle \xi',e_1 \rangle=\langle \xi,e_1 \rangle \quad \text{and} \quad \xi'=\alpha e_1 + \beta e_2, $$ for some $\alpha, \beta \in \mathbb{C}$. 
         
       Then, we have
       \begin{align*}
     \int_{\mathbb{B}^n} \big| \langle \xi, \zeta \rangle  \big| \frac{(1 -|\zeta|^2)^{\alpha}}{\big| 1 - r\zeta_1  \big|^{2\alpha}}\d v(\zeta) &=  
      \int_{\mathbb{B}^n} \big| \langle U\xi, U\zeta \rangle  \big| \frac{(1 -|\zeta|^2)^{\alpha}}{\big| 1 - r\zeta_1  \big|^{2\alpha}}\d v(\zeta) \\ 
     &= \int_{\mathbb{B}^n} \big| \langle \xi', U\zeta \rangle  \big| \frac{(1
       -|\zeta|^2)^{\alpha}}{\big| 1 - r\zeta_1  \big|^{2\alpha}}\d v(\zeta).
     \end{align*}
     Introducing in the last integral a substitution $\zeta = U^*\eta $, we get
      $$ \int_{\mathbb{B}^n} \big| \langle \xi', U\zeta \rangle  \big| \frac{(1 -|\zeta|^2)^{ \alpha}}{\big| 1 - r \langle \zeta,e_1 \rangle   \big|^{2\alpha}}\d v(\zeta)=   
       \int_{\mathbb{B}^n} \big| \langle \xi', \eta \rangle  \big| \frac{(1 -|\eta|^2)^{\alpha}}{\big| 1 - r\eta_1  \big|^{2\alpha}}\d v(\eta), $$
      because 
      $$ \langle U^*\eta, e_1 \rangle = \langle \eta, Ue_1 \rangle = \langle
      \eta, e_1 \rangle= \eta_1. $$
      
     So, the maximum is already attained on vectors of the form $(\xi_1,\xi_2, 0,\dots,0).$ 
     For a given  $\xi=(\xi_1,\xi_2, 0,..,0) $ on the  unit sphere we can make substitution  $\zeta= A^*\eta $, where $A\xi=e_1$, $A$ is an appropriate unitary matrix. 
     We have
     \begin{align*}
     \int_{\mathbb{B}^n} \big| \langle \xi, \zeta \rangle  \big| \frac{(1 -|\zeta|^2)^{\alpha}}{\big| 1 - r\zeta_1  \big|^{2\alpha}}\d v(\zeta) &= 
      \int_{\mathbb{B}^n} \big| \langle A\xi, A\zeta \rangle  \big| \frac{(1 -|\zeta|^2)^{\alpha}}{\big| 1 - r\zeta_1  \big|^{2\alpha}}\d v(\zeta) \\
                                                                                                                                              &=\int_{\mathbb{B}^n} \big| \langle e_1, \eta \rangle  \big| \frac{(1 -|\eta|^2)^{\alpha}}{\big| 1 - r \langle A^*\eta, e_1 \rangle  \big|^{2\alpha}}\d v(\eta) \\
                                                                                                                                              &= 
       \int_{\mathbb{B}^n} |\eta_1| \frac{(1-|\eta|^2)^{\alpha}}{\big|1-r\eta_1 \chi_1 - r\eta_2 \chi_2 \big|^{2\alpha}}\d v(\eta),
       \end{align*}
      where $\big|\chi_1 \big|^2+\big|\chi_2 \big|^2=1,$ which proves lemma.
      \end{proof}
     
      Now, by Fubini's theorem:
      \begin{multline*}
      \int_{\mathbb{B}^n} |\zeta_1|
      \frac{(1-|\zeta|^2)^{\alpha}}{\big|1-r\zeta_1 \xi_1 - r\zeta_2 \xi_2
        \big|^{2\alpha}}\d v(\zeta)= \\
      = \int_{\mathbb{B}^{n-2}} \left(
      \int_{\sqrt{1-|\zeta'|^2}\mathbb{B}^2}|\zeta_1|
      \frac{(1-|\zeta'|^2-|\zeta_1|^2-|\zeta_2|^2)^{\alpha}}{\big|1-r\zeta_1
        \xi_1 - r\zeta_2 \xi_2 \big|^{2\alpha}}\d v(\zeta_1,\zeta_2)\right) \ d v(\zeta_3,\dots,\zeta_n),
      \end{multline*}
      where $\zeta=(\zeta_1,\zeta_2,\zeta').$
      
In the inner integral $I(\zeta',\xi_1, \xi_2,r)$ we use polar coordinates
\begin{gather*}
  \zeta_1=\rho_1e^{i\varphi_1}\sqrt{1-|\zeta'|^2}, \quad
  \zeta_2=\rho_2e^{i\varphi_2}\sqrt{1-|\zeta'|^2}, \\
  D = \left\{ (\varphi_1, \varphi_2, \rho_1, \rho_2)| \rho_1^2+\rho_2^2 <1, \,
    \rho_1,\rho_2>0, \, \varphi_1, \varphi_2 \in [0,2\pi] \right\},
\end{gather*}
and obtain
$$ I(\zeta',\xi_1, \xi_2,r)=
(1-|\zeta'|^2)^{\alpha+ \frac{5}{2}}\int_D \frac{(1-\rho_1^2-\rho_2^2)^\alpha
  \rho_1^2 \rho_2 \d \varphi_1 \d \varphi_2 \d \rho_1 \d \rho_2}{\big|
  1-\sqrt{1-|\zeta'|^2}(r\rho_1\xi_1e^{i\varphi_1}+r\rho_2\xi_2e^{i\varphi_2})
  \big|^{2\alpha}}.$$
      
Using a power series expansion
$$(1 - z)^{-\alpha} = 1 + \sum_{k=1}^{+\infty} \binom{\alpha+k-1}{k} z^{k}, \quad
|z|<1, \quad z \in \mathbb{C}$$ and Parseval's identity, we have:
\begin{multline*}
  \int_{0}^{2\pi}\int_{0}^{2\pi}\frac{\d \varphi_1\d \varphi_2}{\big|1-\sqrt{1-|\zeta'|^2}(r\rho_1\xi_1e^{i\varphi_1}+r\rho_2\xi_2e^{i\varphi_2})
    \big|^{2\alpha}} = \\
  =4\pi^2\sum_{k=0}^{+\infty}{\binom{\alpha+k-1}{k}}^2 r^{2k}(1-|\zeta'|^2)^k
  \sum_{j=0}^{k} \binom{k}{j}^2 (\rho_1|\xi_1|)^{2j}
  (\rho_2|\xi_2|)^{2k-2j}.
\end{multline*}
      
Next, we integrate over the set $\left\{(\rho_1,\rho_2)| \rho_1^2+\rho_2^2<1, \,
  \rho_1,\rho_2>0 \right\}$ and with $|\xi_1|=\cos\theta, \, |\xi_2|=\sin\theta
$, for some $\theta \in [0,\frac{\pi}{2}],$ we get:
\begin{multline*}
    I(\zeta',\xi_1, \xi_2,r) = 4\pi^2 (1-|\zeta'|^2)^{\alpha+ \frac{5}{2}}
    \Bigg( \sum_{k=0}^{+\infty}{\binom{\alpha+k-1}{k}}^2 r^{2k}(1-|\zeta'|^2)^k
    \cdot {} \\
    {} \cdot \bigg( \sum_{j=0}^{k} \binom{k}{j}^2
    \cos^{2j}\theta\sin^{2k-2j}\theta 
    \cdot \int\limits_{\substack{\rho_1^2+\rho_2^2<1\\\rho_1,\rho_2>0}}(1-\rho_1^2-\rho_2^2)^{\alpha}
    \rho_1^{2j+2} \rho_2^{2k-2j+1}\d \rho_1\d \rho_2 \bigg) \Bigg).
\end{multline*}
    
Change of variables $ \rho_1=\sqrt{s}, \, \rho_2=\sqrt{t}$ combined with
Fubini's theorem, and then a new substitution $t=u(1-s)$ gives us
\begin{align*}
  \int\limits_{\substack{\rho_1^2+\rho_2^2<1\\\rho_1,\rho_2>0}} (1-\rho_1^2-&\rho_2^2)^{\alpha}
  \rho_1^{2j+2} \rho_2^{2k-2j+1} \d \rho_1\d \rho_2 = \\
  &= \frac{1}{4}\int\limits_{\substack{s+t<1\\s,t>0}}(1-s-t)^{\alpha} s^{j+\frac{1}{2}} t^{k-j}\d s\d t \\
  &=\frac{1}{4}\int_{0}^{1}\left(\int_{0}^{1-s} (1-s-t)^{\alpha}  t^{k-j}\d t
    \right)s^{j+\frac{1}{2}} \d s\d t  \displaybreak[0]\\ 
                                                                            &=\frac{1}{4}\int_{0}^{1}\left(\int_{0}^{1} (1-s)^{\alpha+k-j+1}(1-u)^\alpha u^{k-j}\d u \right)s^{j+\frac{1}{2}} dsdt \displaybreak[0]\\
                                                                            &=\frac{1}{4} \mathrm{B}(\alpha+1,k-j+1)\mathrm{B}(\alpha+k-j+2,j+\frac{3}{2}) \displaybreak[0]\\
  &=\frac{1}{4} \frac{\Gamma(\alpha+1)\Gamma(k-j+1)}{\Gamma(\alpha+k-j+2)} \frac{\Gamma(\alpha+k-j+2)\Gamma(j+\frac{3}{2})}{\Gamma(\alpha+k+\frac{7}{2})} \\
  &=\frac{1}{4} \frac{\Gamma(\alpha+1)}{\Gamma(\alpha+k+\frac{7}{2})}
  \Gamma(k-j+1)\Gamma(j+\frac{3}{2}),
\end{align*}
so,
\begin{align*}
  I(\zeta',\xi_1, \xi_2,r) &= \pi^2 \Gamma(\alpha+1) \Bigg( \sum_{k=0}^{+\infty}\frac{{\binom{\alpha+k-1}{k}}^2}{\Gamma(\alpha+k+\frac{7}{2})} r^{2k}(1-|\zeta'|^2)^{k+\alpha+ \frac{5}{2}} \cdot {} \\
                               &\hspace{2em}{} \cdot \sum_{j=0}^{k} \binom{k}{j}^2 \Gamma(k-j+1)\Gamma(j+\frac{3}{2})\cos^{2j}\theta\sin^{2k-2j}\theta \Bigg).
\end{align*}
       
       Let us prove that
       $$a_{j,k}=\binom{k}{j} \Gamma(k-j+1) \Gamma\left(j + \frac{3}{2}\right),$$ is, for a  fixed $k,$ increasing in $j,  0\leq j \leq k.$ Indeed
       This is equal, respectively to:
       \begin{multline*}
        a_{j,k}=\frac{k!}{j! (k-j)!} (k-j)! \Gamma \left(j + \frac{3}{2} \right) = \frac{k!}{j!} \Gamma\left( j + \frac{3}{2} \right) = \\ 
        = \frac{k!}{\Gamma(j+1)} \Gamma \left( j + \frac{3}{2} \right) = k! \Gamma \left( \frac{1}{2} \right) \mathrm{B} \left( j+1, \frac{1}{2} \right)^{-1},
       \end{multline*}
      but
      $$\mathrm{B}\left(j+1, \frac{1}{2}\right) = \int_0^1 t^{j+2} (1-t)^{\frac{3}{2}} \, dt$$
      decreases in $j$, so $a_{j,k}$ increases.
      
      This implies
      \begin{multline*}
      \sum_{j=0}^k \binom{k}{j}^2 \Gamma(k-j+1) \Gamma \left(j + \frac{3}{2}\right) \cos^{2j} \theta \sin^{2k - 2j} \theta \leq \\
      \leq \sum_{j=0}^k \binom{k}{j} \Gamma \left(k + \frac{3}{2}\right) \cos^{2j} \theta \sin^{2k - 2j} \theta = \\
      = \Gamma \left( k + \frac{3}{2}\right) \left( \cos^2 \theta + \sin^2 \theta \right)^k
      = \Gamma \left( k + \frac{3}{2}\right),
     \end{multline*}
     and we have the equality for $\theta = 0$, or $\xi_1 = 1, \, \xi_2 = 0$.
     \par Note that this means that in the second integral in $(8)$ the supremum is also attained 
     at $\xi=e_1,$ which makes our estimates sharp. 
     
From the above calculations we can deduce that
\begin{multline*}   
  T(r)(1-r^2)^{-\alpha}=\int_{\mathbb{B}^n} |\zeta_1| \frac{\left( 1 - |\zeta|^2\right)^\alpha
  }{\left| 1 - r\zeta_1 \right|^{2 \alpha}} \d v(\zeta) = \\
  = \pi^2 \sum_{k=0}^{+\infty} \binom{\alpha + k - 1}{k}^2 \frac{\Gamma(\alpha
    +1) \Gamma \left( k + \frac{3}{2} \right)}{\Gamma \left( k + \alpha +
      \frac{7}{2} \right)} r^{2k} \int_{\mathbb{B}^{n-2}} \left(1 - |\zeta'|^2
  \right)^{k + \alpha + \frac{5}{2}} \d v(\zeta').
\end{multline*}

Next, we have:
\begin{multline*}
  \int_{\mathbb{B}^{n-2}} (1 - |\zeta'|^2)^{k + \frac{5}{2} + \alpha} \d
  v(\zeta') = \frac{2\pi^{n-2}}{\Gamma(n-2)} \int_{0}^{1} (1 - r^2)^{k+
    \frac{5}{2}+ \alpha}
  r^{2n-5} \d r = \\
  = \frac{\pi^{n-2}}{\Gamma(n-2)} B(k+\alpha + \frac{7}{2}, n-2) =\pi^{n-2}
  \frac{\Gamma(k+\alpha + \frac{7}{2})}{\Gamma(k+\alpha+n+\frac{3}{2})}.
\end{multline*}

Therefore:
\begin{multline*}
  (n+1) c_\alpha \int_{\mathbb{B}^n} |\zeta_1| \frac{\left( 1 - |\zeta|^2\right)^\alpha
  }{\left| 1 - r\zeta_1 \right|^{2 \alpha}} \d v(\zeta) = \\
  = (n+1)
  \sum_{k=0}^{+\infty} \binom{\alpha + k - 1}{k}^2 \frac{\Gamma(n+\alpha + 1)
    \Gamma{\left( k + \frac{3}{2} \right)}}{\Gamma{\left( k + \alpha +n
        +\frac{3}{2} \right)}} r^{2k}.
\end{multline*}
  (Here we use the values of $c_\alpha$ and the last integral.)
  Observe that the above calculations are valid for $n \geq 2$. 
  The case $n=1$ is much easier and we leave details to the reader.

  \par Now, we prove the following lemma:
  
  \begin{lemma}
    The sequence
$$a_{k} = \binom{\alpha + k-1}{k} \frac{\Gamma(k+\frac{3}{2})}{\Gamma(k+\alpha +
  n+ \frac{3}{2})}, \quad k \in \mathbb{N}_{0} $$
is monotone decreasing for $\alpha \leq 2n +3 $, and for $\alpha > 2n+3$ it
increases for $k\leq k_{\alpha}$ and decreases for $k > k_{\alpha}$.
\end{lemma}
\begin{proof}
  For $k \geq 0$ we have
  \[
    \frac{a_k}{a_{k+1}} = \frac{\binom{\alpha + k -1}{k} \Gamma(k+\frac{3}{2})
      \Gamma(k+\alpha + n+ \frac{5}{2})}{\binom{\alpha+k}{k+1} \Gamma(k+ \alpha
      + n + \frac{3}{2}) \Gamma(k+ \frac{5}{2})} =\frac{(k+1)(k+\alpha + n +
      \frac{3}{2})}{(\alpha+k)(k+\frac{3}{2})}.
  \]

  Now, it is easy to see that $\frac{a_k}{a_{k+1}}>1$ iff
  $2k(n+1) + 2n+ 3 \geq \alpha.$
  
This holds for all $k\geq 0$ iff $\alpha \leq 2n +3$. If $\alpha > 2n+3$, then
for $k = 0$, inequality does not hold, but $2k(n+1)+2n+3$ increases as a
function of $k$, so for some $k=k_{\alpha}$ it will be not smaller than
$\alpha$.

Thus, in case $\alpha> 2n+3$, for $k<k_{\alpha}$ we have $a_k < a_{k+1}$, and
$a_k \geq a_{k+1}$, when $k\geq k_{\alpha}$. Here
$$k_{\alpha} = \left\lceil \frac{\alpha - (2n+3)}{2n+2} \right\rceil. $$

We can also conclude that for $\alpha> 2n+3$, the greatest term in this sequence is 
$$a_{k_{\alpha}} = \binom{\alpha+ k_{\alpha} -
  1}{k_{\alpha}}\frac{\Gamma(k_{\alpha} + \frac{3}{2})}{\Gamma(k_{\alpha} +
  \alpha + n + \frac{3}{2})}. \qedhere$$
\end{proof}

With all these computations and lemmas , we can complete the proof of our Theorem.

Namely, if $\alpha \leq 2n+3$ then 
$$\binom{\alpha + k - 1}{k} \frac{\Gamma(k+ \frac{3}{2})}{\Gamma(k+ \alpha + n +
  \frac{3}{2})} \leq \frac{\Gamma(\frac{3}{2})}{\Gamma(k+ \alpha + n +
  \frac{3}{2})}, \quad  k\geq 0$$
and thus:
\begin{multline*}
(1-r^2)^{\alpha} \sum_{k=0}^{+\infty} 
  (n+1){\binom{\alpha+k-1}{k}}^2\frac{\Gamma(n+
  \alpha+1)\Gamma(k+\frac{3}{2})}{\Gamma(k+\alpha+n+\frac{3}{2})}r^{2k}  \leq \\
\leq(n+1) (1-r^2)^{\alpha}
\sum_{k=0}^{+\infty}{\frac{\Gamma(n+\alpha+1)\Gamma(\frac{3}{2})}{\Gamma(\alpha+
    n + \frac{3}{2})}}\binom{\alpha+k-1}{k} r^{2k} = \\
= (n+1) \frac{\Gamma(n+ \alpha+1)\Gamma(\frac{3}{2})}{\Gamma(\alpha + n + \frac{3}{2})}(1 - r^2)^{\alpha}(1-r^2)^{-\alpha} 
= (n+1) \frac{\Gamma(n+ \alpha+1)\Gamma(\frac{3}{2})}{\Gamma(\alpha + n + \frac{3}{2})}
\end{multline*}
or 
$$(1 - r^2)^{\alpha} \sum_{k=0}^{+\infty} \frac{\Gamma(k+ \frac{3}{2})}{\Gamma(k+ n+\alpha+ \frac{3}{2})}{\binom{\alpha+k-1}{k}}^2 r^{2k}
\leq \frac{\Gamma(\frac{3}{2})}{\Gamma(n+ \alpha+ \frac{3}{2})}. $$

Both inequalities become equalities for $r=0.$

This proves
$$S(z) \leq (n+1) \frac{\Gamma(n+ \alpha + 1)
  \Gamma(\frac{3}{2})}{\Gamma(\alpha + n + \frac{3}{2})}, \quad
\|f\|_\infty \leq 1.$$
Now, taking $z=0, f(w)= \dfrac{|w_1|}{\overline{w_1}} \in L^{\infty},$ for $w =
(w_1, w_2,\dots, w_n)$, we get
\begin{multline*}
  (1- |z|^2)|\nabla_z (B_{\alpha}f)(z)| = |\nabla_z (B_{\alpha}f)(0)|= \\
= \sup_{ |\xi|=1} (n+1)c_{\alpha} \left|\int_{\mathbb{B}^n}
\langle\overline{w} \frac{|w_1|}{\overline{w_1}}, \xi \rangle
(1-|w|^2)^{\alpha}\d (w)\right| \geq \\
 \geq (n+1)c_{\alpha} \int_{\mathbb{B}^n} |w_1|(1-|w|^2)^{\alpha} \d v(w)=
(n+1)\frac{\Gamma(n+ \alpha +1)
  \Gamma(\frac{3}{2})}{\Gamma(\alpha+n+\frac{3}{2})} ,
\end{multline*}
where a choice
$\xi=e_1$ justifies the inequality.
\par Thus, we have
$$\sup_{|z|<1}(1 - |z|^2)|\nabla_{z} (B_{\alpha}f)(z)| = (n+1)\frac{\Gamma(n+ \alpha+1)\Gamma(\frac{3}{2})}{\Gamma(\alpha+n+\frac{3}{2})} = \frac{n+1}{2}B(n+ \alpha+1,\frac{1}{2}).$$

For $\alpha> 2n+3$ we have boundedness but this formula for the norm does not hold. Namely, if we denote
$$c_k = \binom{\alpha+k-1}{k} \frac{\Gamma(k+\frac{3}{2})}{\Gamma(k+ \alpha + n
  + \frac{3}{2})} \quad \text{and} \quad
 d_k = \binom{\alpha+k-1}{k} \frac{\Gamma(\frac{3}{2})}{\Gamma(\alpha+n+\frac{3}{2})},$$
it is easy to see that $c_0 = d_0$, and $c_1>d_1$. 
Then
\begin{multline*}
  (n+1) \Gamma(n+\alpha + 1)(1-r^2)^{\alpha} \sum_{k=0}^{+\infty}
  \binom{\alpha + k-1}{k} \frac{\Gamma(k+\frac{3}{2})}{\Gamma(k+ \alpha+ n +
    \frac{3}{2})}r^{2k} \leq \\
  \leq (n+1) \Gamma(n+ \alpha+1)\frac{\Gamma(\frac{3}{2})}{\Gamma(\alpha
    + n + \frac{3}{2})}
\end{multline*}
is equivalent to
$$\sum_{k=1}^{+\infty} c_k r^{2k} \leq \sum_{k=1}^{+\infty} d_k r^{2k}.$$(observe $c_0 = d_0$).
But,  
$$ \lim_{r\to 0+} \frac{\sum_{k=1}^{+\infty} c_k r^{2k}}{\sum_{k=1}^{+\infty} d_k r^{2k}}=\frac{c_1}{d_1}>1,$$
so the above inequality can not hold for small values of $r$.

We cannot expect, also, that it will be achieved for $r\to 1-$, because
\begin{align*}
  T(r) &= (1-r^2)^{\alpha}\int_{\mathbb{B}^n} |\zeta_1|\frac{(1-|\zeta|^2)^{\alpha}}{|1 - r\zeta_1|^{2\alpha}} \d v(\zeta) \\
          &\leq \int_{\mathbb{B}^n}|\zeta_1|{ \bigg(\frac{(1-r^2)(1 - |\zeta_1|^2)}{|1-r\zeta_1|^2}\bigg) }^{\alpha}\d v(\zeta) \\
  &= \int_{\mathbb{B}^n} |\zeta_1|(1- |\phi_r (\zeta_1)|^2)^{\alpha} \d v(\zeta),
  \end{align*}
but
$$|\zeta_1| (1 - |\phi_r(\zeta_1)|^2)^{\alpha} \leq |\zeta_1| \in L^{1}(\mathbb{B}^n, \d v(\zeta)) $$
and, by Lebesgue's dominated convergence theorem
$$0 \leq \limsup_{r\to 1-} T(r) \leq \int_{\mathbb{B}^n} \lim_{r\to 1-} |\zeta_1|(1 - |\phi_r (\zeta_1)|^2)^{\alpha} \d v(\zeta) = 0$$
as $ \phi_r(\zeta_1) = \frac{r - \zeta_1}{1- r\zeta_1} \to 1$ as $r \to 1-$.
So, the maximum is not attained at $r=0$ or $r=1$. 
\par In this case, we can estimate norm using observation in Lemma 2. Namely,
$$(n+1)c_\alpha T(r)< (n+1)\Gamma(n+ \alpha +1) \binom{\alpha+ k_{\alpha} - 1}{k_{\alpha}} \frac{\Gamma(k_{\alpha} + \frac{3}{2})}{\Gamma (k_{\alpha} + \alpha + n + \frac{3}{2})}. $$
\par The inequality is strict, because, for $r=0$ we already proved it, while for $r>0$:
$a_k < a_{k_{\alpha}}$ is strict for $k\neq k_{\alpha}$.

\par If $\alpha<0 $, the same choice of $f$ as above, brings us, again, to a
real function $T(r), \, 0 \leq r<1 $, but now this function is the product of
the two monotone increasing functions: the first one is $(n+1)c_\alpha(1-r^2)^\alpha$ and the
second one is $ \int_{\mathbb{B}^n} |\zeta_1|\frac{(1-|\zeta|^2)^{\alpha}}{|1 -
  r\zeta_1|^{2\alpha}} \d v(\zeta) $, which is, by the earlier computations also
increasing in $r$. Now, because the first of these is unbounded, we can conclude
that $B_\alpha$ is not bounded for negative $\alpha$.  

  It is straightforward to verify that the analogous sharp estimate holds for  $|\nabla_{\overline{z}} (B_\alpha f)(z)|$ with extremal function $f(w)= \frac{|w_1|}{w_1} \in L^{\infty}.$ 
  
  \section{Proof of the Theorem 2}
  
   During this section we will use the same notation for complex and real scalar product, but it will be clear from the context which one is used. Let us note that the above estimates do not give the sharp constant in the appropriate inequality for real gradient. Instead, we use the inequality:
 \begin{multline}
 |\nabla (B_{\alpha}f)(z)| = \sup_{ |l|=1}
|\langle \nabla (B_{\alpha}f)(z), l \rangle| \\ 
= \sup_{|l|=1}
\left| \int_{\mathbb{B}^n} \langle \nabla K (z,w)f(w), l  \rangle \d v_{\alpha}(w) \right| \\
 \leq \sup_{ |l|=1} \int_{\mathbb{B}^n}
|\langle \nabla K(z,w), l \rangle| \, |f(w)|
\d v_{\alpha} (w) \\
\leq \sup_{ |l|=1} \int_{\mathbb{B}^n}
|\langle \nabla K(z,w), l \rangle| \, 
\d v_{\alpha} (w)  \|f\|_{\infty}
\end{multline}
Here $l \in \mathbb{R}^{2n}.$ 

Also, we need certain connections with previous calculations to find the supremum from the last expression. Namely, since
$$ \frac{\partial}{\partial x_k}K(z,w)= \frac{\partial}{\partial z_k}K(z,w)+\frac{\partial}{\partial {\overline z}_k}K(z,w) $$
and
$$ \frac{\partial}{\partial y_k}K(z,w)= i \bigg(\frac{\partial}{\partial z_k}K(z,w)-\frac{\partial}{\partial {\overline z}_k}K(z,w) \bigg), $$
where $z_k=x_k+i y_k, k=1,2 \dots, n,$ we have:
\begin{align*}
	\langle \nabla K(z,w), l \rangle &= \sum_{k=0}^{n}  \frac{\partial}{\partial x_k}K(z,w)l_{2k-1} + \frac{\partial}{\partial y_k}K(z,w)l_{2k}\\
&= \sum_{k=0}^{n} \frac{\partial}{\partial z_k}K(z,w) (l_{2k-1}+i l_{2k})
+ \frac{\partial}{\partial {\overline z}_k}K(z,w) (l_{2k-1}-i l_{2k})=\\
&2 Re \sum_{k=0}^{n} \frac{\partial}{\partial z_k}K(z,w) \overline{\xi_k}= 2 Re \langle \nabla_{z}K(z,w), \xi \rangle, 
\end{align*}
for $\xi_k=l_{2k-1}-i l_{2k},$ because kernel $K(z,w)$ is real valued.

Let us assume that $\|f\|_{\infty}=1.$
Hence, by the earlier computations, we obtained:

$$ (1- |z|^2)|\nabla (B_{\alpha}f) (z)| \leq 2(n+1)c_{\alpha} \sup_{ |\xi|=1 }\int_{\mathbb{B}^n}
 \frac{(1-|z|^2)^{n+1}}{|1- \langle z, w \rangle|^{2n+2}} \cdot {} $$\\ 
$$ {} \cdot \big|Re  \frac {(1- |z|^2)\langle \xi , w \rangle - (1 - \langle z , w
 \rangle)\langle \xi, z \rangle}{1- \langle z, w \rangle}  \big| \d v_{\alpha} (w)= $$\\ 
$$  2(n+1) c_\alpha \sup_{|\xi|=1} \int_{\mathbb{B}^n}
\big| Re  \langle \xi , \frac{\langle \zeta, z\rangle}{|z|^2}z + \sqrt{1 - |z|^2} (\zeta - \frac{\langle \zeta, z \rangle}{|z|^2}z)\rangle \big| \frac{(1- |z|^2)^{\alpha}(1- |\zeta|^2)^{\alpha}}{|1 - \langle z. \zeta  \rangle|^{2\alpha} } \d v(\zeta)  $$

Again, assuming $z=(r,0 \dots 0), 0 \leq r < 1$  we get:

\begin{multline*}
	(1-|z|^2)|\nabla (B_{\alpha}f)(z)| \leq 2(n+1) c_{\alpha}(1-r^2)^{\alpha} \\
\times 	\max_{|\xi|=1}  \int_{\mathbb{B}^n} \left| Re  \langle \xi, \sqrt{1-r^2}\zeta + (1-\sqrt{1-r^2})\zeta_1^{'} \rangle \right| \frac{(1- |\zeta|^2)^{\alpha}}{{|1-r\zeta_1|^{2\alpha}}} \d v(\zeta).
\end{multline*}
Recall that $\zeta_1'=\zeta_1e_1$ for $\zeta=(\zeta_1, \dots, \zeta_n).$

Triangle inequality, as in (7), gives us:
\begin{align}
	(1-|z|^2)\left|\nabla (B_{\alpha}f)(z)\right| &\leq 2(n+1) c_{\alpha} \bigg[ (1-\sqrt{1-r^2})\nonumber\\
	 &\quad\times (1-r^2)^{\alpha} \max_{|\xi|=1} \int_{\mathbb{B}^n } \left|Re \langle \xi, \zeta_1{'}  \rangle \right| \frac{(1- |\zeta|^2)^{\alpha}}{{|1-r\zeta_1|^{2\alpha}}} \d v(\zeta) \label{nejednakost}\\
 &\quad+ (1-r^2)^{\alpha+\frac{1}{2}} \max_{|\xi|=1} \int_{\mathbb{B}^n } \left|Re \langle \xi, \zeta  \rangle \right| \frac{(1- |\zeta|^2)^{\alpha}}{{|1-r\zeta_1|^{2\alpha}}} \d v(\zeta) \bigg].\nonumber
\end{align}
 
The first maximum is attained for $\xi=e_1$, while for the second we again use the similar argument from Lemma 1 to get:

\begin{multline*}
	\max_{|\xi|=1}  \int_{\mathbb{B}^n}| Re \langle \xi, \zeta \rangle | \frac{(1- |\zeta|^2)^{\alpha}}{{|1-r\zeta_1|^{2\alpha}}} dv(\zeta)= \\
	=\max_{|\xi_1|^2+|\xi_2|^2=1 }  \int_{\mathbb{B}^n}| Re  \zeta_1 | \frac{(1-
		|\zeta|^2)^{\alpha}}{{|1-r\zeta_1\xi_1-r\zeta_2\xi_2|^{2\alpha}}}
	\d v(\zeta).
\end{multline*}
  
Now, we estimate

$$ \int_{\mathbb{B}^n} | Re \zeta_{1} | \frac{(1-|\zeta|^2)^{\alpha}}{|1-r\zeta_{1}\xi_1-r\zeta_{2}\xi_2 |^{2\alpha}} \d v(\zeta).$$

Denote  $\zeta_j=\mu_{2j-1}+i\mu_{2j}$, $\xi_j=\nu_{2j-1}+i\nu_{2j}$. Then

$$ \zeta_1\xi_1 + \zeta_2\xi_2=\mu_1\nu_1-\mu_2\nu_2+ \mu_3\nu_3-\mu_4\nu_4+i(\mu_1\nu_2+\mu_2\nu_1+\mu_3\nu_4+\mu_4\nu_3) $$

and so:
$$ \int_{\mathbb{B}^n} | Re \zeta_{1} | \frac{(1-|\zeta|^2)^{\alpha}}{|1-r\zeta_{1}\xi_1-r\zeta_{2}\xi_2 |^{2\alpha}} \d v(\zeta) \leq$$
$$ \int_{\mathbb{B}_{2n}} | \mu_{1} | \frac{(1-|\mu|^2)^{\alpha}}{|1-r(\mu_{1}\nu_1-\mu_{2}\nu_2+\mu_{3}\nu_3-\mu_{4}\nu_4) |^{2\alpha}} \d v(\mu).$$
Here $\mu=(\mu_1,\mu_2,\dots,\mu_{2n})$ and $\d v(\mu)$ is $2n-$dimensional Lebesgue measure.
We have used the inequality $|z| \geq |Re z|.$

In the last integral we can introduce variables which rotate $(-\nu_2,\nu_3,-\nu_4)$ to $\sqrt{\nu_2^2+\nu_3^2+\nu_4^2}(1,0,0)$ and hence we have:
$$ \int_{\mathbb{B}^n} | Re \zeta_{1} | \frac{(1-|\zeta|^2)^{\alpha}}{|1-r\zeta_{1}\xi_1-r\zeta_{2}\xi_2 |^{2\alpha}} \d v(\zeta) \leq$$
$$ \int_{\mathbb{B}_{2n}} | \mu_{1} | \frac{(1-|\mu|^2)^{\alpha}}{|1-r\mu_{1}\nu_1-r\mu_{2}\nu_2 |^{2\alpha}} \d v(\mu).$$
where $\nu_1^2+\nu_2^2 = 1, \nu_1,\nu_2 \in \mathbb{R}.$

Using series expansion for $(1-r\mu_{1}\nu_1-r\mu_{2}\nu_2)^{-2\alpha}$ we obtain:

$$  \int_{\mathbb{B}_{2n}} | \mu_{1} | \frac{(1-|\mu|^2)^{\alpha}}{\big(1-r\mu_{1}\nu_1-r\mu_{2}\nu_2 \big)^{2\alpha}} \d v(\mu)=$$
$$  \int_{\mathbb{B}_{2n}} | \mu_{1} | (1-|\mu|^2)^{\alpha} \sum_{k=0}^{+\infty} \binom{k+2\alpha-1}{k} r^k \sum_{j=0}^{k} \binom{k}{j} \mu_1^{k-j}\nu_1^{k-j} \mu_2^{j}\nu_2^j \d v(\mu)  = $$
$$ \sum_{k=0}^{+\infty} \binom{k+2\alpha-1}{k} r^k \sum_{j=0}^{k} \binom{k}{j} \nu_1^{k-j}\nu_2^j \int_{\mathbb{B}_{2n}} | \mu_{1} | (1-|\mu|^2)^{\alpha}  \mu_1^{k-j}\mu_2^j \d v(\mu). $$

It is obvious that for $j$ or $k-j$ odd, the integral over $\mathbb{B}_{2n}$ is zero, so the last sum is equal to:

$$ \sum_{k=0}^{+\infty} \binom{2k+2\alpha-1}{2k} r^{2k} \sum_{j=0}^{k} \binom{2k}{2j} \nu_1^{2k-2j}\nu_2^{2j} \int_{\mathbb{B}_{2n}} | \mu_{1} | (1-|\mu|^2)^{\alpha}  \mu_1^{2k-2j}\mu_2^{2j} \d v(\mu). $$
By Fubini's theorem:

$$\int_{\mathbb{B}_{2n}} | \mu_{1} |^{2k-2j+1}\mu_2^{2j} (1-|\mu|^2)^{\alpha}   \d v(\mu)= $$
$$\int_{\mathbb{B}_{2}} | \mu_{1} |^{2k-2j+1}\mu_2^{2j} \int_{\sqrt{1-\mu_1^2-\mu_2^2}\mathbb{B}_{2n-2}} (1-\mu_1^2-\mu_2^2-\mu'^2)^{\alpha}   \d v(\mu') \d v(\mu_1,\mu_2)= $$
$$\int_{\mathbb{B}_{2n-2}} (1-\tau^2)^{\alpha} \d v(\tau) \int_{\mathbb{B}_{2}} | \mu_{1} |^{2k-2j+1}\mu_2^{2j}(1-\mu_1^2-\mu_2^2)^{\alpha+n-1} \d \mu_1 \d \mu_2 =$$
$$\frac{\pi^{n-1}\Gamma(\alpha+1)}{\Gamma(\alpha+n)} \int_{u+v \leq 1, u,v\geq 0} u^{k-j} v^{j-\frac{1}{2}} (1-u-v)^{\alpha+n-1} \d u \d v=$$
$$\frac{\pi^{n-1}\Gamma(\alpha+1)}{\Gamma(\alpha+n)} B(k-j+1,\alpha+n) B(\alpha+n+k-j+1,j+\frac{1}{2})=$$
$$\frac{\pi^{n-1}\Gamma(\alpha+1)}{\Gamma(\alpha+n)}\frac{\Gamma(\alpha+n)}{\Gamma(\alpha+n+k+\frac{3}{2})}\Gamma(k-j+1)\Gamma(j+\frac{1}{2})=$$
$$\frac{\pi^{n-1}\Gamma(\alpha+1)}{\Gamma(\alpha+n+k+\frac{3}{2})}\Gamma(k-j+1)\Gamma(j+\frac{1}{2})=$$

Here we have used change of variables and then calculations of some integrals and very familiar relation between Gamma and Beta functions.

Now, we estimate the double sum:
$$\sum_{k=0}^{+\infty} \binom{2k+2\alpha-1}{2k} \frac{1}{\Gamma(\alpha+n+k+\frac{3}{2})} r^{2k} \sum_{j=0}^{k} \binom{2k}{2j} \Gamma(k-j+1)\Gamma(j+\frac{1}{2}) \nu_1^{2k-2j}\nu_2^{2j}.$$
 
 Let us note that sequence $b_{j,k}=\frac{\binom{2k}{2j}}{\binom{k}{j}}\Gamma(k-j+1)\Gamma(j+\frac{1}{2})$ is decreasing in $j$ for fixed $k$.
 Indeed, from
 $$b_{j,k}=\frac{\binom{2k}{2j}}{\binom{k}{j}}\Gamma(k-j+1)\Gamma(j+\frac{1}{2})=\frac{(2k)!j!(k-j)!^2}{(2j)!(2k-2j)!k!}\Gamma	(j+\frac{1}{2})$$
 we get
$$ \frac{b_{j+1,k}}{b_{j,k}}=\frac{(j+1)!(k-j-1)!^2\Gamma(j+\frac{3}{2})(2j)!(2k-2j)!}{(2j+2)!(2k-2j-2)!j!(k-j)!^2\Gamma(j+\frac{1}{2})}=\frac{2k-2j-1}{2k-2j} <1$$	
for $0 \leq j \leq k-1.$

Hence we obtain the estimate:

$$ \sum_{k=0}^{+\infty} \binom{2k+2\alpha-1}{2k} \frac{1}{\Gamma(\alpha+n+k+\frac{3}{2})} r^{2k} \sum_{j=0}^{k} \binom{2k}{2j} \Gamma(k-j+1)\Gamma(j+\frac{1}{2}) \nu_1^{2k-2j}\nu_2^{2j} \leq $$
$$ \sum_{k=0}^{+\infty} \binom{2k+2\alpha-1}{2k} \frac{1}{\Gamma(\alpha+n+k+\frac{3}{2})} r^{2k} \sum_{j=0}^{k} \binom{k}{j} k!\Gamma(\frac{1}{2}) \nu_1^{2k-2j}\nu_2^{2j} = $$
$$ \sum_{k=0}^{+\infty} \binom{2k+2\alpha-1}{2k} \frac{1}{\Gamma(\alpha+n+k+\frac{3}{2})} k!\Gamma(\frac{1}{2}) r^{2k} \big(\nu_1^2+\nu_2^2\big)^k =$$
$$\sum_{k=0}^{+\infty} \binom{2k+2\alpha-1}{2k} \frac{1}{\Gamma(\alpha+n+k+\frac{3}{2})} k!\Gamma(\frac{1}{2}) r^{2k}$$

It remains, now, to find the supremum of this function on $r \in [0,1].$ From the inequality \eqref{nejednakost} and estimates of the double sum, it is clear that this estimate would be sharp only if the supremum is achieved for $r=0.$  We will find the relation between $n$ and $\alpha$ for which it holds, and give an estimate for the other case. We need the following:

 \begin{lemma}
	The sequence
	$$b_{k} = \binom{2\alpha +2 k-1}{2k}\binom{\alpha + k-1}{k}^{-1}\frac{k!}{\Gamma(k+\alpha +
		n+ \frac{3}{2})}, \quad k \in \mathbb{N}_{0} $$
	is monotone decreasing in  $k$ iff $\alpha \leq n+\frac{1}{2}$, and for $\alpha>n+\frac{1}{2}$ it
	increases for $k\leq k'_{\alpha}$ and decreases for $k > k'_{\alpha}$.
\end{lemma}
\begin{proof}
	For $ b_k= \frac{k!\binom{2k+2\alpha-1}{2k}}{\Gamma(k+n+\alpha+\frac{3}{2})\binom{k+\alpha-1}{k}}$
we easily calculate 
$$ \frac{b_{k+1}}{b_{k}}=\frac{(2k+2\alpha+1)(k+1)}{(2k+1)(k+n+\alpha+\frac{3}{2})},$$
which is $\leq 1$ iff $\alpha \leq n+\frac{1}{2}+k(2n+1).$
This holds for all $k \geq 0,$ iff $\alpha \leq n+\frac{1}{2}$. In case $\alpha>n+\frac{1}{2},$
the sequence $b_k$ increases in $k$, for $k \leq k'_{\alpha}$ and decreases for $k \geq k'_{\alpha}$, where
$$ k'_{\alpha}= \left\lceil \frac{\alpha}{2n+1}-\frac{1}{2} \right\rceil $$
So, for $ \alpha \leq n+\frac{1}{2}$ the sequence is decreasing and $b_k\leq b_0$ for all $ k\in \mathbb{N},$ while  for $\alpha >n+\frac{1}{2}$, we have 
$$b_k \leq b_{k'_\alpha}=\binom{2k'_{\alpha}+2\alpha-1}{2k'_{\alpha}} {\binom{k'_{\alpha}+\alpha-1}{k'_{\alpha}}}^{-1} \frac{(k'_{\alpha})!}{\Gamma(k'_{\alpha}+\alpha+n+\frac{3}{2})}.$$

\end{proof}

As a consequence of the previous Lemma, we conclude:

For $\alpha \leq n+\frac{1}{2}$ all the above estimates is sharp and we have
$$(1-|z|^2)\left|\nabla (B_{\alpha}f)(z)\right| \leq 2(n+1)c_{\alpha}\pi^{n-1}\frac{\Gamma(\alpha+1)\Gamma(\frac{1}{2})}{\Gamma(\alpha+n+\frac{3}{2})}\|f\|_{\infty}=$$
$$\frac{2}{\pi}(n+1) B(\alpha+n+1,\frac{1}{2})\|f\|_{\infty}.$$

For $\alpha > n+\frac{1}{2}$ we have the estimate:
\begin{multline*}
	(1-|z|^2)\left|\nabla (B_{\alpha}f)(z)\right| < \frac{2}{\sqrt{\pi}}(n+1)\binom{2k'_{\alpha}+2\alpha-1}{2k'_{\alpha}}\\ {\binom{k'_{\alpha}+\alpha-1}{k'_{\alpha}}}^{-1} \frac{(k'_{\alpha})!\Gamma(\alpha+n+1)}{\Gamma(k'_{\alpha}+\alpha+n+\frac{3}{2})} \|f\|_{\infty}.
\end{multline*}	

It is important to say that choosing $z=(1,0,\dots,0)$ and $r=0$ we achieve the equality in the first estimate for $f(z)=\frac{|Re z_1|}{Re z_1}.$ Also, it is clear that we have

\par \textbf{Acknowledgements}. I would like to thank prof.~dr Miloš Arsenović, for his useful comments and
helpful discussions and dr Marijan Marković for the suggestion of the problem.

\end{document}